%% file: main.tex
\documentclass[a4paper, 11pt,reqno]{amsart}
\usepackage[margin=2cm]{geometry}
\usepackage{amsthm,amssymb,amsmath,graphicx,psfrag,enumitem,mathrsfs}
\usepackage{mathtools}
\usepackage[foot]{amsaddr}
\usepackage[british]{babel}
\usepackage[babel]{microtype}

\usepackage[textsize=footnotesize]{todonotes}

\usepackage{tikz}
\usepackage{setspace}

\usetikzlibrary{backgrounds,arrows.meta,arrows,shapes}

\usepackage[numbers]{natbib}

\makeatletter
\def\NAT@spacechar{~}
\makeatother

\usepackage{hyperref}
\usepackage[capitalise]{cleveref}
\hypersetup{colorlinks=true,
   citecolor=blue,
   filecolor=blue,
   linkcolor=blue,
   urlcolor=blue
}

\crefname{figure}{Figure}{Figures}
\Crefname{figure}{Figure}{Figures}

\allowdisplaybreaks

\newtheorem{definition}{Definition}[section]
\newtheorem{claim}{Claim}

\newtheorem{proposition}[definition]{Proposition}
\newtheorem{theorem}[definition]{Theorem}

\newtheorem{lemma}[definition]{Lemma}

\newtheorem{conjecture}[definition]{Conjecture}
\newtheorem{question}[definition]{Question}

\numberwithin{equation}{section}

\newcommand{\comment}[1]{}

\newcommand{\cF}{\mathcal{F}}

\newcommand{\cP}{\mathcal{P}}
\newcommand{\cQ}{\mathcal{Q}}
\newcommand{\cR}{\mathcal{R}}
\newcommand{\cS}{\mathcal{S}}
\newcommand{\cT}{\mathcal{T}}

\newcommand{\ora}{\overrightarrow}
\newcommand{\init}{\text{init}}
\newcommand{\ter}{\text{ter}}
\newcommand{\inter}{\text{int}}

\renewcommand{\epsilon}{\varepsilon}

\newcommand\Clemmabranches{10^4k^{3}}
\newcommand\Clemmapaths{10^{13}k^{4}}
\newcommand\CWibound{12k^{22}\ell^2}
\newcommand\Clbound{3k + \Clemmabranches{} + 2\cdot\Clemmapaths{}}

\newcommand{\COMMENT}[1]{}
\renewcommand{\COMMENT}[1]{\footnote{\textcolor{blue!70!black}{#1}}} 

\newcounter{kccounter}

\title{$(2k+1)$-connected tournaments with large minimum out-degree are $k$-linked}   

\author[A.~Gir\~{a}o]{Ant\'onio Gir\~{a}o}
\address[Ant\'onio Gir\~{a}o]{School of Mathematics, University of Birmingham, 
Edgbaston, Birmingham, B15 2TT, United Kingdom.
}
\email{giraoa@bham.ac.uk}

\author[K.~Popielarz]{Kamil Popielarz}
\email{kamil.popielarz@gmail.com}

\author[R.~Snyder]{Richard Snyder}
\address[Richard Snyder]{Karlsruhe Institute of Technology, Karlsruhe, Germany.}
\email{richard.snyder@kit.edu}

\thanks{The first author wishes to acknowledge support by the EPSRC, grant. no. EP/N019504/1.}

\date{\today}

\begin{document}
\onehalfspacing

\begin{abstract}
 Pokrovskiy conjectured that there is a function $f: \mathbb{N} \rightarrow \mathbb{N}$ such that any $2k$-strongly-connected tournament with minimum out and in-degree at least $f(k)$ is $k$-linked. In this paper, we show that any $(2k+1)$-strongly-connected tournament with minimum out-degree at least some polynomial in $k$ is $k$-linked, thus resolving the conjecture up to the additive factor of $1$ in the connectivity bound, but without the extra assumption that the minimum in-degree is large. Moreover, we show the condition on high minimum out-degree is necessary by constructing arbitrarily large tournaments that are $(2.5k-1)$-strongly-connected but are not $k$-linked. 

\end{abstract}
\maketitle

\onehalfspacing
\section{Introduction}

This paper is concerned with the relation between two central notions of connectivity in tournaments: strong-connectivity and linkedness.
A directed graph is \emph{strongly-connected} if for any pair of distinct vertices $x$ and $y$ there is a directed path from $x$ to $y$, and is strongly $k$-connected if it has at least $k + 1$ vertices and if it remains strongly-connected upon the removal of any set of at most $k-1$ vertices. We shall omit the use of the word
`strongly' with the understanding that we always mean strong connectivity.  A directed graph $G$ is $k$-\emph{linked} if $|V(G)| \ge 2k$ and
for any two disjoint sets of vertices $\{x_1, \ldots , x_k\}$ and $\{y_1, \ldots, y_k\}$ there are pairwise vertex disjoint directed 
paths $P_1, \ldots , P_k$ such that $P_i$ has initial vertex $x_i$ and terminal vertex $y_i$ for every $i \in [k]$. Thus, $G$ 
is $1$-linked if and only if it is connected and nontrivial. Since linkedness is a stronger notion than connectivity, it is natural to ask if a high enough connectivity is sufficient to guarantee linkedness. This is too much to hope for in general, as shown by Thomassen~\cite{Thomassen_construct}, who constructed digraphs of arbitrarily large connectivity, but which are not even $2$-linked. This is in stark contrast to the situation for undirected graphs: Bollob{\'a}s and Thomason~\cite{BollobasThomason} showed that any $22k$-connected graph is $k$-linked and using this result they confirmed a conjecture of Mader~\cite{MaderConjecture} and of Erd\H{o}s and Hajnal~\cite{ErdosHajnal} related to the smallest average degree that guarantees a subdivision of a clique on $k$ vertices (this result was also proved independently by Koml{\'o}s and Szemer{\'e}di~\cite{KomlosSzemeredi}). The constant in the connectivity bound in Bollob{\'a}s and Thomason's result has subsequently been improved by Thomas and Wollan~\cite{ThomasWollan}. They showed that $2k$-connectivity suffices for $k$-linkedness provided the graph has average degree at least $10k$. Thomassen~\cite{Thomassen-2-linked} conjectured that any $(2k+2)$-connected graph is $k$-linked, though this is false if the graph does not have sufficiently many vertices (for example, $K_{3k-1}$ minus a matching of size $k$ is a counterexample. It is $(3k-3)$-connected, but not $k$-linked). If the amended conjecture is true, it would in fact be tight (see~\cite{Kawarabayashi-et-al}).

The picture in the directed setting is more positive, however, if we restrict our attention to \emph{tournaments}, those directed graphs obtained by orienting each edge of a complete graph in precisely one direction. Indeed, Thomassen~\cite{Thomassen_tourn} was the first to find a function $g(k)$ such that any $g(k)$-connected tournament is $k$-linked. The initial bounds on this function were poor: Thomassen proved the above result with $g(k) = 2^{Ck\log k}$, but there came a series of two major improvements. First, K\"uhn, Lapinskas, Osthus, and Patel~\cite{Lap_Kuhn_Osthus_Patel} proved that it suffices to take $g(k) = 10^4 k\log k$ and conjectured that one could remove the logarithmic factor. Pokrovskiy~\cite{Pokrovskiytourn}, resolving this conjecture, showed that any $452k$-connected tournament is $k$-linked. Kang and Kim~\cite{KangKim} proved an extension of this result, namely, that there exists an absolute constant $C$ such that any $Ck$-connected tournament is $k$-linked, where the paths witnessing $k$-linkedness have prescribed lengths (provided the lengths are sufficiently large). Pokrovskiy went on to conjecture that one could push `$452$' down to `$2$' as long as the tournament has large minimum in/out-degree:

\begin{conjecture}[Pokrovskiy~\cite{Pokrovskiytourn}]\label{conj:pokrovskiy}
There is a function $f(k)$ such that any $2k$-connected tournament $T$ with $\delta^0(T) \ge f(k)$ is $k$-linked, where $\delta^0(T) = \min\{\delta^+(T), \delta^-(T)\}$.
\end{conjecture}

This conjecture may be viewed as a directed analogue of several results in the undirected setting. In particular, as was mentioned earlier, $2k$-connectivity suffices provided one imposes some density condition on the graph (like large average degree). Here, the natural `density' condition for a tournament is large minimum in/out-degree.

In \Cref{section:constructions}, we construct two families of tournaments that demonstrate that both conditions in \Cref{conj:pokrovskiy} are necessary: we show that for every $k \geq 2$ there exist infinitely many tournaments that are $(2k-1)$-connected with  arbitrarily large minimum in/out-degrees, but which are not $k$-linked. Additionally, for every even $k \geq 6$ there are infinitely many tournaments that are $(2.5k-1)$-connected but are not $k$-linked. 

The first and last authors~\cite{girao-snyder} proved that the statement of Pokrovskiy's conjecture holds with `$2k$' replaced by `$4k$', without the assumption of large minimum in-degree.

\begin{theorem}[Gir\~{a}o, Snyder~\cite{girao-snyder}]\label{thm:old-main}
There is a function $f(k)$ such that any $4k$-connected tournament $T$ with $\delta^+(T) \ge f(k)$ is $k$-linked.
\end{theorem}

Our main result improves the above result in two ways, and nearly resolves \Cref{conj:pokrovskiy} in a stronger form. First, we are able to reduce the connectivity bound to $2k+1$. Second, we only impose a condition on the minimum out-degree, and the bound we obtain is significantly better than that obtained in \Cref{thm:old-main}. More precisely, while we proved \Cref{thm:old-main} with $f(k)$ doubly-exponential, our main result shows that we make take $f(k)$ to be a polynomial.

\begin{theorem}\label{thm:main}
There exists a polynomial $f$ such that any $(2k+1)$-connected tournament $T$ with $\delta^+(T) \ge f(k)$ is $k$-linked.
\end{theorem}

An analysis of our proof shows that we may take $f(k) = Ck^{31}$ for some sufficiently large (but absolute) constant $C$.
We have not made an attempt to optimize the power of $k$ (see \Cref{sec:final_remarks}).

Our proof of \Cref{thm:main} requires the notion of a subdivision. Recall that the complete digraph on $k$ vertices, denoted by $\ora{K}_k$, is a directed graph in which every pair of vertices is connected by an edge in each direction. As usual, we say that a tournament $T$ contains a subdivision of $\ora{K}_k$ if it contains a set $B$ of $k$ vertices and a collection of $2\binom{k}{2}$ pairwise internally vertex disjoint directed paths joining every ordered pair of vertices in $B$. We denote such a subdivision by $T\ora{K}_k$, and the vertices in $B$ are called the \emph{branch vertices} of the subdivision. Further, for a positive integer $\ell$ we denote by $T_\ell\ora{K}_k$ a subdivision of $\ora{K}_k$ where each edge is replaced by a directed path of length at most $\ell + 1$ (i.e., each edge is subdivided at most $\ell$ times).

A central tool in the proof of \Cref{thm:old-main} was to show that tournaments with sufficiently high minimum out-degree contain subdivisions of complete digraphs. In particular, we showed that there is an absolute constant $c$ such that any tournament $T$ with $\delta^+(T) \ge 2^{2^{ck^2}}$ contains a $T\ora{K}_k$. Recently, the authors~\cite{GiraoSnyderPopielarz} improved considerably this result, reducing the bound on minimum out-degree to a quadratic.

\subsection{Notation and organization}
Our notation is standard. Thus, for a vertex $v$ in a directed graph $G$, we let $N_G^+(v), N_G^-(v)$ denote the out-neighbourhood and in-neighbourhood of $v$, respectively. Moreover, we let $d_G^+(v) = |N_G^+(v)|$ denote the out-degree of $v$, and analogously $d_G^-(v)$ the in-degree of $v$. We often omit the subscript `$G$' when the underlying digraph is clear. We denote by $\delta^+(G)$ the minimum out-degree of $G$; further, if $X \subset V(G)$, we write $\delta^+(X)$ to mean the minimum out-degree of $G[X]$. For a subset $X \subset V(G)$ we let $N^+(X)$ denote the set $\bigcup_{x \in X}N^+(x)$. If $X, Y \subset V(G)$, we write $X \rightarrow Y$ if every edge of $G$ between $X$ and $Y$ is directed from $X$ to $Y$. Whenever $X = \{x\}$ we simply write $x \rightarrow Y$ (and similarly if $Y = \{y\}$). If $P = x_1\ldots x_\ell$ is a directed path, then we refer to $x_1$ as the \emph{initial} vertex of $P$, and say that $P$ \emph{starts} at $x_1$. Similarly, we call $x_\ell$ the \emph{terminal} vertex of $P$, and say that $P$ \emph{ends} at $x_\ell$. We refer to both $x_1$ and $x_\ell$ as \emph{endpoints} of $P$. The subpath of $P$ excluding the initial and terminal vertices of $P$ is called the \emph{interior} of $P$, denoted by $\inter(P)$.

The remainder of this paper is organized as follows.
In Section~\ref{sec:mainresult}, we give the proof of our main theorem.
In Section~\ref{section:constructions}, we present two families of constructions showing the necessity of both conditions in Theorem~\ref{thm:main}.
Finally, we close the paper with some open problems in Section~\ref{sec:final_remarks}.

\section{Proof of the main result}\label{sec:mainresult}

\subsection{Preliminaries}\label{subsec:overview}

We need the following simple lemma from~\cite{GiraoSnyderPopielarz}. To state it, we say that a subset $B \subset V(T)$ is $C$-\emph{nearly-regular} if either $d^-(v) \le d^+(v) \le Cd^-(v)$ for every $v \in B$, or $d^+(v) \le d^-(v) \le Cd^+(v)$ for every $v \in B$. Further, $B$ is $(C, m, t)$-\emph{nearly-regular} if it is $C$-nearly-regular and additionally $d^-(v) \in [m - 10t, m + 10t]$ for every $v \in B$. The following lemma allows us to find $(4, m, t)$-nearly-regular $t$-element subsets in tournaments. We include the short proof for the reader's convenience.

\begin{lemma}\label{lem:ratios}
Any tournament $T$ contains a $4$-nearly-regular subset of size $|T|/10$, and a $(4, m, t)$-nearly-regular subset of size $t$ provided $|T|\geq t$, for some $m$.
 \end{lemma}
 
\begin{proof}
We first claim that $T$ contains a $4$-nearly-regular subset of size at least $|T|/10$. Indeed, let $|T| = n$ and let $R \subset V(T)$ be the vertices for which either the ratio between the out-neighbourhood and in-neighbourhood (or vice-versa) is between $1$ and $4$. If $|R|\geq n/5$, then we are done, as we may pass to a subset $A \subset R$ of at least half the size for which the property is satisfied for one or the other. 
If not, then let $T'=T\setminus R$, so that $|T'|\geq 4n/5$. Let $T'_1$ be the set of vertices $v \in V(T')$ for which $d_T^+(v) > 4d_T^-(v)$ and $T'_2$ be those vertices $v \in V(T')$ for which $d_T^-(v) > 4d_T^+(v)$. Suppose without loss of generality that $|T'_1|\geq |T'_2|$, so that $|T'_1|\geq 2n/5$. This implies that there is a vertex $u$ in $T'_1$ which has in-degree inside $T'_1$ at least $n/5$. But then
\[
    n/5 \leq d_T^-(u) < \frac{1}{4}d_T^+(u) \le n/5,
\]
a contradiction.

Thus, we can always find a $4$-nearly-regular subset $A$ of size at least $|T| / 10$. Partition the interval $[1, \ldots, |T|]$ into consecutive intervals of size $10t$, and distribute the vertices of $A$ according to their in-degrees in $T$. By the pigeonhole principle, there must exist at least
\[
    10t\cdot \frac{|A|}{|T|} \ge 10t\cdot \frac{1}{10} = t
\]
vertices in the same interval. These $t$ vertices form a $(4, m, t)$-nearly-regular subset for some $m$.
\end{proof}

We also need the following well-known result by Erd\H{o}s and Szekeres.

\begin{proposition}\label{prop: Erdos-Szekeres}
Let $S$ be a finite set of order $n$ and suppose there exist $\ell$ total orderings $<_1, \ldots, <_{\ell}$ on $S$. Then there exists a subset $S'\subset S$ of size at least $t=n^{1/2^{\ell-1}}$ and an ordering of $S'=\{s_1,s_2,\ldots,s_t\}$ such that, 
in this ordering, $S'$ forms an increasing chain in $<_1$ and an increasing chain or decreasing chain in $<_i$ for every $i\in \{2,\ldots, \ell\}$.
\end{proposition}

\subsection{Finding a $(k, \ell)$-good family}\label{subsec:structure}
We need some terminology to state our main lemma precisely. 
Let $T$ be a tournament and let $X \subset V(T)$. We say that a subdivision contained in $T$ \emph{sits on} $X$ if its branch vertex set is some subset of $X$. 

\begin{definition}\label{def:good-family}
Let $T$ be a tournament and suppose $k, \ell$ are positive integers. A family $\cF$ of pairwise disjoint subsets of $V(T)$ is $(k,\ell)$-\emph{good} in $T$ if it satisfies the following properties: 
\begin{itemize}
\item For each $A \in \cF$, if there is no $T_2\ora{K}_\ell$ siting on $A$, then $|A|= 12k^2$.

\item For each distinct $A, B \in \cF$, if there is no $T_2\ora{K}_\ell$ siting on $A$ and no $T_2\ora{K}_\ell$ sitting on $B$, then either
\[
    A \rightarrow B \,\,\text{ or }\,\, B \rightarrow A \,\, \text{ in }\,\, T.
\]
\end{itemize}
Moreover, we denote by $S(\cF)$ the \emph{subdivision sets} in $\cF$: those $A \in \cF$ such that there exists a $T_2\ora{K}_\ell$ sitting on $A$.
\end{definition}


The first part of the definition is there for technical reasons later in the proof. The important point is that the size of $A$ is quadratic in $k$. 

Here is our main lemma.

\begin{lemma}\label{lem:structure}
Let $k \le \ell$ be positive integers and suppose $T$ is a tournament such that $V(T) = \bigcup_{i \in [k]} W_i$ for pairwise disjoint subsets $W_1,W_2,\ldots, W_k$ with $|W_i| \ge \CWibound{}$ for each $i$.
Then there exists a family of sets $\cF = \{S_1, \ldots, S_k\}$ that is $(k, \ell)$-good in $T$ with $S_i \subset W_i$ for $i = 1, \ldots, k$. Furthermore, there is a family $\cT$ of pairwise vertex disjoint copies of $T_2\ora{K}_\ell$ such that for each $S_i \in S(\cF)$ there is some subdivision $\cS \in \cT$ that sits on $S_i$.
\end{lemma}

\begin{proof}
We can assume that each $W_i$ has size precisely $\CWibound{}$ by, when necessary, passing to smaller subsets of exactly that size. 
We proceed by induction on $k$; for $k = 1$ there is nothing to show, so assume $k \ge 2$ and the result holds for smaller values. Applying \Cref{lem:ratios} with $t = k\ell$ to $T$, we find a $(4, m, k\ell)$-nearly-regular subset $A \subset \bigcup_{i\in [k]}W_i$ of size $k\ell$. Without loss of generality, we may assume that $d^-(v) \le d^+(v) \le 4d^-(v)$ for every $v \in A$.
Now, there is a subset $B$ of $A$ size at least $|A|/k = \ell$ contained in some $W_i$, say $W_1$, without loss of generality. We now break the proof up into two cases, depending on whether or not there is a $T_2\ora{K}_{\ell}$ sitting on $B$. Suppose there exists such a subdivision sitting on $B$, say $\mathcal{S}$. Then remove the non-branch vertices of $\cS$ from $T$ to form the subtournament $T' = \bigcup_{i = 2}^k W'_i$. We have removed at most $2\binom{\ell}{2} < \ell^2$ vertices from $T$, so $|W'_i| > \CWibound{} - \ell^2 \ge 12(k-1)^{22}\ell^2$, and so we are done by induction applied to $T'$ .

Therefore, assume that no subdivision sits on $B$. In order to apply induction, we require the following claim, which allows us to partition the $W_i$'s in a particularly nice way.
\begin{claim}\label{claim:inductive-partition}
Let $W_1, \ldots, W_k$ and $B$ be as above with no $T_2\ora{K}_{\ell}$ sitting on $B$. Then there is a partition $I \cup J = [k]$ and families $\cF_1 = \{W_i : i \in I\}$, $\cF_2 = \{W_j : j \in J\}$ satisfying:
\begin{enumerate}
    \item \label{itm:claim1-1} $|\mathcal{F}_i|\geq k/10$ for $i = 1, 2$.
    \item \label{itm:claim1-2} There exists $W'_i\subset W_i$ with $|W'_i| \ge |W_i|/10$ for each $i \in [k]$ such that
        \[
            \bigcup_{i\in I} W_i' \rightarrow \bigcup_{j \in J} W_j'.
           \] 
\end{enumerate}
\end{claim}

\begin{proof}
Suppose we try to embed greedily a $T_2\ora{K}_\ell$ with branch vertex set $B$. Since by assumption we cannot succeed, there exists a partial subdivision $\cS$ and two distinct vertices $x, y \in B$ such that
\[
    N^+(x) \cap N^-(y) \subset \cS \,\, \text{ and }\,\, N^-(y) \setminus \cS \rightarrow N^+(x)\setminus \cS.
\]
As $B$ is $(4, m, k\ell)$-nearly-regular for some $m$ we have that $d^-(x), d^-(y) \in [m - 10k\ell, m + 10k\ell]$. Further, $|N^-(y) \setminus N^-(x)| = |N^+(x) \cap N^-(y)| \le |\cS| \le 2\binom{\ell}{2} = \ell^2$, so we obtain
\[
    |N^-(x) \setminus N^-(y)| \le |N^-(y) \setminus N^-(x)| + 10k\ell \le \ell^2 + 10k\ell \le 11\ell^2,
\]
where the last inequality follows since $k \le \ell$, by assumption. Recall that $B$ is $4$-nearly-regular, and as such, $d^-(v) \le d^+(v) \le 4d^+(v)$ for every $v \in B$. This implies, in particular, that $d^-(v), d^+(v) \ge |T|/5$ for each $v$. Letting $X = N^+(x) \setminus \cS$ and $Y = N^-(y) \setminus \cS$, we obtain that $|X|$ and $|Y|$ are both at least
\begin{align*}
    &\ge|T|/5 - |\cS| - |N^+(x) \cap N^-(y)| - |N^-(x) \cap N^+(y)|\\
    &\ge |T|/5 - \ell^2 - \ell^2 - 11\ell^2\\
    &= |T|/5 - 13\ell^2.
\end{align*}
Moreover, $|X \cup Y| \ge |T| - 12\ell^2$; without loss of generality, assume $|Y| \ge |X|$. Then $|Y| \ge |T|/2 - 6\ell^2$, and $|X| \ge |T|/5 - 13\ell^2$.
In summary, we have obtained large sets $X$ and $Y$ with $Y \rightarrow X$, and such that $X\cup Y$ covers most of $T$. In particular, as $V(T) = \bigcup_{i \in [k]} W_i$,  for each $i \in [k]$ we have that either $|W_i \cap Y| \ge |W_i|/2 - 6\ell^2 \ge |W_i|/10$ or $|W_i \cap X| \ge |W_i|/2 - 6\ell^2 \ge |W_i|/10$.
Now, partition $[k]$ into $I$ and $J$ such that
    for every $i \in I$ we have $|W_i \cap X | \geq |W_i|/10$,
    for every $j \in J$ we have $|W_j \cap Y | \geq |W_j|/10$,
    and $|I|, |J|$ are as equal as possible.
Finally, set $\cF_1 = \{W_i: i \in I\}$ and $\cF_2 = \{W_j: j \in J\}$ and let $W_i' = W_i \cap Y$ if $i \in I$ and $W_i' = W_i \cap X$ if $i \in J$. Then property (\ref{itm:claim1-2}) of the claim certainly holds by definition of the sets $W_i'$ and the fact that $Y \rightarrow X$, so we just need to check that $|I|$ and $|J|$ are large according to (\ref{itm:claim1-1}).
Suppose for contradiction that this is not the case and, say, $|I| < k / 10$.
This means that $|J| > \frac{9}{10}k$ and for every $j \in J$ we have $|X \cap W_j| < |W_j|/10$, as otherwise we could move $j$ from $J$ to $I$, decreasing the distance between $|I|$ and $|J|$ and hence contradicting the assertion that $|I|, |J|$ are as equal as possible.
Therefore, $|W_j \cap Y| \geq |W_j| - |W_j \cap X| - 12\ell^2 > \frac{9 |W_j|}{10} - 12 \ell^2$ for every $j \in J$.
It follows that $|Y| > \frac{9}{10}k \cdot (\frac{9|T|}{10k} - 12 \ell^2) \ge \frac{81}{100}|T| - 12 k\ell^2$, which implies that $|X| < \frac{19}{100}|T| + 12k\ell^2$.
Together with the fact that $|X| \geq |T|/5 - 13 \ell^2$, we have
\[
    |T| < 100 (12k\ell^2 + 13 \ell^2) 
        \leq 2500k\ell^2
        < 12 \cdot 2^{21}k\ell^2 
        \leq \CWibound{},
\]
where we used the assumption that $k \geq 2$.
This contradicts the fact that $|T| \geq |W_1| \geq \CWibound{}$.
\end{proof}

Let $W_1', \ldots, W_k'$, and $\cF_1, \cF_2$ be given as in the conclusion of \Cref{claim:inductive-partition}. As $\min\{|\mathcal{F}_1|,|\mathcal{F}_2|\}\geq k/10$, and since for each $i$ we have
\[
    |W'_i|\geq \frac{12k^{22}}{10}\ell^2 \geq 12(9k/10)^{22}\ell^2,
\]
apply induction to $T_1 = T[\bigcup_{i \in I}W_i']$ and to $T_2 = T[\bigcup_{j \in J} W_j']$. This yields a $(k, \ell)$-good family $\cF = \{S_1, \ldots, S_k\}$ in $T$, and a family $\cT$ of copies of $T_2\ora{K}_\ell$, as required. Indeed, if there is no subdivision sitting on $S_i$ nor $S_j$, then if $i, j \in I$ (or $i, j \in J$), the required property is satisfied by induction, and if $i \in I$, $j \in J$, then $S_i \rightarrow S_j$ by construction of the families $\cF_1$ and $\cF_2$. Moreover, as $T_1$ and $T_2$ are disjoint, all subdivisions in $\cT$ are pairwise vertex disjoint.
\end{proof}

\subsection{Utilizing a $(k, \ell)$-good family}\label{subsec:completing}

Let $T$ be a $(2k+1)$-connected tournament with minimum out-degree at least $k \cdot \CWibound{} + 2k$, where $\ell \geq \Clbound{}$.
We may assume that $k \ge 2$, since the result is immediate for $k = 1$. Suppose $X = \{x_1, \ldots, x_k\}$ and $Y = \{y_1, \ldots, y_k\}$ are vertex disjoint $k$-sets of vertices and that we wish to link $x_i$ to $y_i$ in $T$ for each $i \in [k]$. First, because of the large minimum out-degree we can find pairwise vertex disjoint sets $W_i \subset \left( N^+(x_i) \setminus \left(X \cup Y\right) \right)$ such that $|W_i| = \CWibound{{}}$ for each $i = 1, \ldots , k$.
Consider now the subtournament
\[
   T_0 = T[\bigcup_{i \in [k]} W_i].
\]

By \Cref{lem:structure} applied to $T_0$, find a $(k, \ell)$-good family $\cF = \{S_1, \ldots, S_k\}$ with $S_i \subset W_i$ for each $i$, and a family $\cT$ of pairwise disjoint copies of $T_2\ora{K}_\ell$ according to the lemma. Recall that $S(\cF)$ denotes the subdivision sets in $\cF$: those $S_i \in \cF$ such that there exists a subdivision from $\cT$ sitting on $S_i$. We shall assume that each such $S_i$ consists of precisely those branch vertices in the corresponding subdivision that sits on it (by possibly removing some vertices from $S_i$). Further, any set in $\cF$ that is not a subdivision set we shall call a \emph{non-subdivision} set.

Define an auxiliary digraph $H$ with vertex set $[k]$ in the following way. For a pair $i, j \in [k]$ with $i \neq j$:
\begin{itemize}
    \item If $S_i$ and $S_j$ are non-subdivision sets, then orient $i$ to $j$ if $S_i \rightarrow S_j$ in $T$, and vice-versa if $S_j \rightarrow S_i$ (here we are using that $\cF$ is a $(k, \ell)$-good family).
    
    \item If $S_i$ is a non-subdivision set but $S_j$ is, then orient $i$ to $j$ if \hypertarget{first_case}{(a)} at least $|S_i|/2$ vertices in $S_i$ have at least $|S_j|/2$ out-neighbours in $S_j$, and orient $j$ to $i$ if \hypertarget{second_case}{(b)} at least $|S_i|/2$ vertices in $S_i$ have at least $|S_j|/2$ in-neighbours in $S_j$.
    
    \item  If both $S_i$ and $S_j$ are subdivision sets, orient $i$ to $j$ if at least $|S_i|/2$ vertices in $S_i$ have at least $|S_j|/2$ out-neighbours in $S_j$, and orient $j$ to $i$ if at least $|S_j|/2$ vertices in $S_j$ have at least $|S_i|/2$ out-neighbours in $S_i$.
    
\end{itemize}
Note that $H$ is a semicomplete digraph (i.e., it is a tournament with some potential double edges created between subdivision sets in the third case above).
Let $H^s$ denote the subdigraph of $H$ induced on $\{i \in [k]: S_i \text{ is a subdivision set} \}$, and let $H^{ns}$ denote the subdigraph induced on the remaining vertices of $H$. Observe that it is possible that one of $H^s, H^{ns}$ is empty. Since $H^s$ is a tournament up to some possible double edges, it contains a Hamiltonian path $P^s$. Similarly, $H^{ns}$ contains a Hamiltonian path $P^{ns}$. With some abuse of notation we write
\[
    P^s = S_{i_1}\ldots S_{i_r}\, \text{ and }\, P^{ns} = S_{j_1}\ldots S_{j_t},
\]
where $\{i_1, \ldots, i_r\}\cup \{j_1, \ldots, j_t\} = [k]$,
to emphasize that these paths in $H^s$ and $H^{ns}$ correspond to sequences of the sets $S_i$. Further, we shall at times write $\init(P^s)$ and $\ter(P^s)$ for $S_{i_1}$ and $S_{i_r}$, respectively (and similarly for $P^{ns}$).

For technical reasons we \hypertarget{discarded}{discard} the vertices in $S_{j_t}$ which have fewer than $|S_{i_r}|/2$ out-neighbours in $S_{i_r}$ in case \hyperlink{first_case}{(a)}, or fewer than $|S_{i_r}|/2$ in-neighbours in $S_{i_r}$ in case \hyperlink{second_case}{(b)}.

Our aim is to apply Menger's theorem to find $k+1$ vertex disjoint paths from either $\ter(P^s)$ or $\ter(P^{ns})$ to $Y \cup \{v\}$, where $v$ is some vertex in $\cup_{q=1}^t S_{j_q}$ (initially, we choose $v \in \init(P^{ns})$ to have high out-degree in $\init(P^{ns})$). We call the initial set of $k+1$ vertices of these paths the \emph{origin}, and denote it by $O$, and we call the vertex $v$ the \emph{special vertex}. The choice of $O$ depends on the following circumstances:
\begin{enumerate}
    \item If $P^s = \varnothing$, choose $O \subset \ter(P^{ns})$; similarly, if $P^{ns} = \varnothing$, choose $O \subset \ter(P^s)$. \label{itm:origin-1}
    \item If $i_rj_t \in E(H)$, choose $O \subset S_{j_t}$.\label{itm:origin-2}
    \item If $j_ti_r \in E(H)$, choose $O \subset S_{i_r}$.\label{itm:origin-3}
\end{enumerate}
In each case, we initially let the special vertex $v$ be an element of $\init(P^{ns})$ with the largest out-degree in $\init(P^{ns})$, except of course when $P^{ns} = \varnothing$: in that case we choose no special vertex and we let $O \subset \ter(P^s)$ be a set of size $k$. In fact, we shall assume that (\ref{itm:origin-1}) does not occur since the proof in this case follows from the arguments for cases (\ref{itm:origin-2}) and (\ref{itm:origin-3}) (and is simpler).

So choose $O$ in accordance with (\ref{itm:origin-2}) or (\ref{itm:origin-3}) and let $v$ be the special vertex in $\init(P^{ns})$. Since $T$ is $(2k+1)$-connected, Menger's theorem implies that, upon the removal of $X$, there exists a family $\cQ$ of $k+1$ pairwise vertex disjoint directed paths from $O$ to $Y \cup \{v\}$. Let $\cQ$ be chosen to minimize $|\bigcup \cQ|$, the total number of vertices used in the paths. We refer to the path in $\cQ$ ending at $v$ as the \emph{special path} in $\cQ$. In general, during the course of the proof we shall make modifications to the family $\cQ$. If $\cQ'$ denotes another collection of paths from $O$ to $Y \cup \{v'\}$ for some $v' \notin Y$, then we refer to the path ending at $v'$ as the \emph{special path} of $\cQ'$.

We would like to do the following for each $i = 1, \dots, k$: starting with $x_i$, form a directed path to $y_i$ by first choosing an out-neighbour of $x_i$ in $S_i$, then travelling along one of the paths $P^s$ or $P^{ns}$ (depending on whether $S_i$ happens to be a subdivision or non-subdivision set) to the corresponding vertex in $O$. Finally, we use the paths from $\cQ$
to reach $y_i$. We need to ensure these paths are chosen disjointly, but more importantly, we need to ensure that the initial paths in $\cQ$ do not obstruct our goal. In other words, we need to make sure that the paths in $\cQ$ do not intersect the $S_i$'s in too many places. 

Our first lemma in this regard asserts that if $S_i$ is a subdivision set, then the paths from $\cQ$ do not intersect $S_i$ in many places.

\begin{lemma}\label{lem:subdivision-intersect}
Suppose that $S = S_i$ is a subdivision set. Then at most $10^4k^3$ vertices of $S$ belong to $\bigcup \cQ$.
\end{lemma}

\begin{proof}
As $S$ is a subdivision set, there is a copy $\cS$ of $T_2\ora{K}_\ell$ in $T_0$ with branch vertex set $S$. For each ordered pair of vertices $a, b \in S$ write $P_{ab}$ for the path in $\cS$ from $a$ to $b$. Suppose the lemma is false, so that there is some path $Q := Q_j \in \cQ$ which intersects $S$ in $m = 10^4k^3/(k+1) \ge 10^3k^2$ vertices. Denote these vertices by $u_1, \ldots, u_m$ in the ordering they appear along the path $Q$. A subdivision path $P_{ab}$ is \emph{free} if no path of $\cQ$ intersects the interior $\inter(P_{ab})$. In the remainder of this proof, we write $P_{i,j}$ for $P_{u_iu_j}$. If any of the paths in $\cP := \{ P_{i, m-i+1}: i = 1, \ldots, 10^3k^2\}$ are free, then we reach a contradiction with the minimality of $\cQ$: each path $P_{i, j}$ has length at most $3$, so we can replace $Q$ with a shorter path by simply taking a shortcut through one of the free paths. It follows that each of these $10^3k^2$ paths contains at least one vertex from $\bigcup \cQ$.

Since each path in $\cP$ has at most two internal vertices, there are at most $(k + 1)^2$ possible intersection patterns with paths in $\cQ$. Indeed, if $P \in \cP$ is given, then there are $k+1$ choices for each internal vertex: one of the $k$ paths in $\cQ$, or none. By the pigeonhole principle, there is a collection of paths $\cP' \subset \cP$ of size at least $|\cP|/(k+1)^2 = 10^3k^2/(k+1)^2\ge 100$ such that all have the same intersection pattern. Suppose the intersection pattern is given by $M, N \in \cQ$, not both empty.
In other words, $M$ and $N$ intersect the second and third vertex, respectively, of each path in $\cP'$. Our aim now is to create a new family of paths from $O$ to $Y\cup \{v\}$ which uses fewer vertices, contradicting the minimality of $\cQ$. To this end, we shall assume that both $M$ and $N$ are nonempty paths, as the case when one of them is empty follows a similar analysis. We define two total orderings $<_1,<_2$ on the paths in $\cP'$.
Indeed, we say $P_{i,j}<_1 P_{i',j'}$ if the second vertex of $P_{i,j}$ comes before the second vertex of $P_{i',j'}$ in the path $M$, and similarly $P_{i,j}<_2 P_{i',j'}$ if the third vertex of $P_{i,j}$ comes before the third vertex of $P_{i',j'}$ in the path $N$.

Applying \Cref{prop: Erdos-Szekeres}, we may pass to a subcollection $\cP''\subset \cP'$, where  $|\cP''|\geq |\cP'|^{1/4}\geq (100)^{1/4} \ge 3$. Let $\cP''=\{P_{i_1, m-i_1+1},P_{i_2,m-i_2+1},P_{i_3,m-i_3+1}\}$ (where $i_1<i_2<i_3$) such that it forms an increasing or decreasing chain in each of the total orderings $<_1$ and $<_2$. 

There are four cases to consider:
\begin{enumerate}
    \item $\cP''$ forms an increasing chain in both $<_1$ and $<_2$;\label{case:rerouting-sub-1}
    \item $\cP''$ forms an increasing chain in $<_1$ and decreasing chain in $<_2$;\label{case:rerouting-sub-2}
    \item $\cP''$ forms a decreasing chain in $<_1$ and an increasing chain in $<_2$;
    \item $\cP''$ forms a decreasing chain in both $<_1$ and $<_2$.
\end{enumerate}

We shall see how to proceed in Cases (\ref{case:rerouting-sub-1}) and (\ref{case:rerouting-sub-2}). The other cases follow by a symmetric argument.

So first assume that $\cP''$ forms an increasing chain in both $<_1$ and $<_2$.
We are going to perform the following rerouting of the paths $Q,M,N$ using the paths in $\cP''$. For brevity, write $P_{i_j}$ for $P_{i_j, m - i_j + 1}$ for $i = 1, 2, 3$. Form the path $M'$ by following $Q$ to $u_{i_3}$, go to the second vertex of $P_{i_3}$ (which belongs to $M$), and continue via $M$ to the terminal vertex of $M$. Form the path $N'$ by following $M$ from the initial vertex of $M$ to the second vertex of $P_{i_2}$ (which we can do, since this vertex appears in $M$ before the second vertex of $P_{i_3}$), following $P_{i_2}$ to the third vertex of $P_{i_2}$ (which belongs to $N$), and continue via $N$ to the terminal vertex of $N$. Finally, form $Q'$ by following $N$ from its initial vertex to the third vertex of $P_{i_1}$, continue to the last vertex of $P_{i_1}$, and then continue via $Q$ to the terminal vertex of $Q$. It is not hard to check that in this process we gain $3$ new directed edges, but loose at least $5$. Thus, letting $\cQ' = (\cQ \setminus \{Q, M, N\}) \cup \{Q', M', N'\}$, we see that $|\bigcup \cQ'| < |\bigcup \cQ|$, contradicting the minimality of $\cQ$.

Let us consider Case (\ref{case:rerouting-sub-2}), that is, $\cP''$ is an increasing chain in $<_1$ and a decreasing chain in $<_2$. 
Now we may form $Q'$ by following $N$ to the third vertex of $P_{i_3}$, following $P_{i_3}$ to the last vertex of $P_{i_3}$ (which is $u_{m - i_3 + 1} \in Q$), and then continuing along $Q$ to the terminal vertex of $Q$. Form $M'$ by following $Q$ from its initial vertex to $u_{i_2}$, then following $P_{i_2}$ to its second vertex (which is in $M$), and then continuing along $M$ to the terminal vertex of $M$. Lastly, form $N'$ following $M$ from its initial vertex to the second vertex of $P_{i_1}$, following $P_{i_1}$ to its third vertex (which is in $N$), and then continuing via $N$ to the terminal vertex of $N$. As before, we gain $3$ new edges but lose at least $5$, so the path system obtained by replacing $Q$, $M$, $N$ with $Q', M', N'$, respectively, has a fewer total number of vertices, contradicting minimality.
\end{proof}

We also need to show that for each subdivision set, `many' of the subdivision paths connecting branch vertices do not intersect the path system $\cQ$. This is the content of the following lemma.

\begin{lemma}\label{lem: subd-paths-intr}
Let $S = S_i$ be a subdivision set and let $\cS$ denote the subdivision sitting on $S$. Let $x\in S$ with $x\notin \bigcup \cQ$. Then there are at most $10^{13}k^{4}$ paths in $\cS$ with one endpoint $x$ and another endpoint in $S\setminus \bigcup \cQ$ that intersect $\bigcup \cQ$.
\end{lemma}

\begin{proof}
    This proof follows a very similar argument as in the proof of \Cref{lem:subdivision-intersect}. 
    Let ~$\cP=\{P_1,P_2,\ldots, P_{10^{13}k^{4}}\}$ be a collection of paths from the subdivision $\cS$ each of which has non empty intersection with $\bigcup \cQ$. Moreover, suppose each path of $\cP$ has $x\in S\setminus \bigcup \cQ$ as an initial or terminal vertex. 
    
    We may assume at least half of these paths, say $\cP' := \{P_1,\ldots, P_{10^{12}k^{4}}\}$, start at $x$ (the other case is symmetric). Note that each path in $\cP'$ has length at most $3$, which implies that each such path can have at most $2$ vertices that belong to $\bigcup \cQ$. As argued before, there are at most $2(k+1)^2$ possible patterns regarding intersections with $\bigcup \cQ$. We therefore may pass to a sub-collection of $\cP'$ of size at least $|\cP'|/2(k+1)^2 \ge 10^{12}k^{4}/8k^2 > 10^{11}k^2$ where all paths have the same intersection pattern. 
    With a slight abuse of notation we shall still denote this collection of paths by $\cP'$. Suppose the intersection pattern is given by $M,N \in \cQ$, not both empty. In other words, $M,N$ intersect the second and third vertex, respectively, of each path in $\cP'$. In the following we identify a path in $\cP'$ by its second and third vertices. Thus if $P \in \cP'$ has second vertex $a$ and third vertex $b$, then we shall write $P_{ab}$ for $P$. In this case, we have $a \in M$ and $b \in N$ (one of $M, N$ could be empty).
    Now, we apply the same procedure to the paths from $\cS$ that start at the terminal vertices of the paths in $\cP'$ and which end at $x$. Indeed, let this collection of paths be denoted by $\cR'$. As before, since there are at most $2(k+1)^2$ possible intersection patterns between a path in $\cR'$ and the collection $\bigcup \cQ$, we may pass to a sub-collection of $\cR'$, say $\cR''$ all of whose paths have intersection pattern $M',N'\in \cQ$, where now both $M'$ and $N'$ could be empty. Clearly $|\cR''|\geq |\cP'|/2(k+1)^2\geq 10^{10}$. Let $\cP''$ be the set of paths in $\cP$ which end at an initial vertex of some path in $\cR''$. 
    
    We define two total orderings $<_1, <_2$ on the paths of $\cP''$: $P_{ab}<_1P_{cd}$ if $a$ comes before $c$ in the path $M$, and 
     $P_{ab}<_2 P_{cd}$ if $b$ comes before $d$ in the path $N$.
    By Proposition~\ref{prop: Erdos-Szekeres} applied to the two orderings $<_1$, $<_2$, we may pass to a subset of $\cP''$ of size $(10^{10})^{1/2} \geq  10^{5}$ with an ordering of the paths such that they form an increasing chain in $<_1$ and an increasing or decreasing chain in $<_2$. For simplicity, denote this collection again by $\cP''$. 
    
    Likewise, we may define two total orderings on the paths of $\cR''$ (restricted to those paths with the same endpoints as paths in $\cP''$ and with the induced ordering given by $<_1$): define $P_{ab}<_3 P_{cd}$ if $a$ comes before $c$ in the path $M'$, and 
     $P_{ab}<_4 P_{cd}$ if $b$ comes before $d$ in the path $N'$. Applying Proposition~\ref{prop: Erdos-Szekeres} to $\cR''$, with the orderings $<_1,<_3,<_4$, we obtain a collection of paths of size at least $(10^5)^{1/4} > 12$ which forms an increasing chain in $<_1$ and an increasing or decreasing chain in $<_3$ and $<_{4}$. With a slight abuse of notation we shall still denote this sub-collection by $\cR''$. Let $\cR''=\{R_1,R_2,\ldots, R_{12}\}$ and $\cP''=\{P_1,P_2,\ldots, P_{12}\}$ be the corresponding paths in $\cP'$ where the endpoint of $P_i$ is the initial vertex of $R_i$, for every $i\in [12]$. We shall assume that each of the paths $M, N, M', N'$ are nonempty (otherwise, the following rerouting argument only becomes simpler).
     
     There are eight cases to consider, depending on whether or not $\cP''$ is an increasing or decreasing chain in $<_2$, and whether or not $\cR''$ is an increasing or decreasing chain in $<_3$ and $<_4$. We consider two of these cases (the others follow by a similar arguments):
     \begin{enumerate}
        \item $\cP''$ forms an increasing chain in both $<_1$ and $<_2$; $\cR''$ forms increasing chain in both $<_3$ and $<_4$; \label{case:rerouting-sub-path-1}
         \item $\cP''$ is increasing in $<_1$ and decreasing in $<_2$; $\cR''$ is increasing in $<_3$ and increasing in $<_4$. \label{case:rerouting-sub-path-2}
     \end{enumerate}
     
    Let us consider now Case (\ref{case:rerouting-sub-path-1}).
    We are going to make the following rerouting of the paths $M,N,M'$ and $N'$. More precisely, whenever the path $M$ hits $P_4$, then it goes to the third vertex of $P_4$ and continues via the sub-path of $N$ which starts at the third vertex of $P_4$. Whenever the path $N$ hits $P_3$ (which is before hitting $P_4$, by assumption) then it goes to the second vertex of $R_3$ and continues via $M'$. Similarly, the path $M'$ is altered in the following way: whenever it hits $R_2$ (which is before hitting $R_3$) then it goes to the third vertex of $R_2$ and continues via the path $N'$. Finally, whenever $N'$ hits $R_1$ then it goes to $x$ and then to the second vertex of $P_{12}$ which belongs to $M$ and continues via $M$. Call this new collection of paths $\cQ'$. Note that we have added at most $6$ more edges in total by using the paths in $\cP''$ and $\cR''$, but we now miss all second vertices of the paths $P_5,\ldots, P_{11}$. Therefore we decreased $|\bigcup \cQ'|$, which is a contradiction.
    
   Finally, consider Case (\ref{case:rerouting-sub-path-2}). We perform the following rerouting. Whenever $N$ hits the third vertex of $P_3$, follow $P_3$ to its terminal vertex, then to the second vertex of $R_3$, and continues via $M'$. Now, starting from $M'$, follow $M'$ to the second vertex of $R_2$, follow $R_2$ to its third vertex, then continue via $N'$. Starting from $N'$, follow $N'$ to the third vertex of $R_1$, continue along $R_1$ to $x$, follow $P_{12}$ to its second vertex, and then continue via $M$. Finally, starting from $M$, follow $M$ to the second vertex of $P_1$, follow $P_1$ to its third vertex, and then continue via $N$. Denoting this new collection by $\cQ'$, we note that $\cQ'$ uses $6$ new edges from the paths in $\cP'', \cR''$, but avoids the second vertices of $P_2, \ldots, P_{11}$. It follows that $|\bigcup \cQ'| < |\bigcup \cQ|$, a contradiction.
    \end{proof}

Our next goal is to show that our path system $\cQ$ can be modified such that it does not intersect non-subdivision sets in too many places. Roughly speaking, we will show that, if some path in $\cQ$ intersects some non-subdivision set $S_i$ in many places, then we can transform $\cQ$ into another collection of paths that ($1$) are vertex disjoint and go from $O$ to $Y \cup \{v^*\}$, where $v^*$ is some vertex in $\bigcup_{q=1}^t S_{j_q}$, ($2$) does not intersect subdivision sets in more places than $\cQ$ did, and ($3$) intersects non-subdivision sets in `few' vertices. We formalize this in the following lemma. To state it precisely, we make the following definitions. Let $\cP$ be some path system constructed in the above process from $O$ to $Y \cup \{z\}$ with special vertex $z \in \bigcup_{q=1}^t S_{j_q}$. A vertex $s \in S_{j_q}$ is $\cP$-\emph{free} if no path in $\cP$ intersects $s$. Furthermore, we say that $S_{j_q}$ is $(\cP, l)$-\emph{free} if it contains at least $l$ free vertices.

\begin{lemma}\label{lem:non-sub-intersect}
There exists a family $\cQ^*$ of vertex disjoint directed paths from $O$ to $Y \cup \{v^*\}$, where $v^* \in \bigcup_{q = 1}^t S_{j_q}$ satisfying the following properties:
\begin{enumerate}
    \item $S_{j_q}$ is $(\cQ^*, q)$-free for $1 \le q \le t$.
    \item The paths in $\cQ^*$ do not intersect subdivision sets in more vertices than paths in $\cQ$ do.
\end{enumerate}
\end{lemma}
\begin{proof}
We consider the sets $S_{j_q}$ in order from $q = 1$ to $t$, and show that we can incrementally free vertices in each set along the way. The process terminates with the desired path system $\cQ^*$. To simplify notation during the course of this proof, we write $S_q$ for $S_{j_q}$ for $q = 1, \ldots, t$. To begin, consider $S_1$ together with the original path system $\cQ$ from $O$ to $Y \cup \{v\}$ with $v \in S_1$. We may assume that $v$ was chosen to be a vertex in $S_1$ with out-degree at least $(|S_1) - 1)/2$ in $S_1$. Since the $S_q$'s are a part of a $(k, \ell)$-good family, we have $|S_q| = 12k^2$ (see \Cref{def:good-family}). Moreover, we discarded at most half of the vertices of the last non-subdivision set $S_{t}$, so $|S_t| \ge 12k^2/2 = 6k^2$. Now, if $\bigcup \cQ$ intersects $S_1$ in at most $|S_1| - 2$ vertices, then $S_1$ is $(\cQ, 2)$-free. Otherwise, $\bigcup \cQ$ intersects $S_1$ in at least $|S_1| - 1$ vertices, and hence intersects $N := N^+_{T_1}(v)$ in at least $|N| - 1 \ge |S_1|/2 - 1 \ge 3k^2 - 1$ vertices. Thus some path $P \in \cQ$ intersects $N$ in at least $3k^2 - 1/(k + 1) > 2k - 1$ vertices. Let $u_1, \ldots, u_l$ be the vertices in the intersection in their order along $P$ with $l \ge 2k$. If $P$ is the special path in $\cQ$, then replace $P$ with $P' = Pu_1$, so that $u_1$ is the new special vertex. Then the vertices $u_2, \ldots, u_l$ are free. Otherwise, $P$ is not a special path. Let $Q \neq P$ denote the special path in $\cQ$ with terminal vertex $v \in S_1$. Replace $Q$ with $Pu_1$ and let $u_1$ be the new special vertex. Replace $P$ with the path $P'$ defined by following $Q$ to $v$, going along the edge $vu_l$, then following $P$ to its endpoint in $Y$. We have thus freed vertices $u_2, \ldots u_{l-1}$, and since $l \ge 2k \ge 4$, we have freed at least $2$ vertices. 

In any case, we denote the resulting collection of paths and special vertex by $\cQ^1$ and $v^1 \in S_1$, respectively. Observe that $S_1$ is now $(\cQ^1, 2)$-free.

Now, suppose $2 \le p < t$, and that we have already constructed a family of paths $\cQ^{p-1}$ with special vertex $v^{p-1} \in S_z$ where $1 \le z \le p-1$ satisfying the following properties:
\begin{itemize}
    \item $S_q$ is $(\cQ^{p-1}, q)$-free for all $2 \le q \le z-1$.
    \item $S_q$ is $(\cQ^{p-1}, q+1)$-free for all $z \le q \le p-1$.
\end{itemize}

We show how to construct $\cQ^p$. If the $k+1$ paths in $\cQ^{p-1}$ intersect $S_p$ in less than $5k^2$ vertices, then, recalling that each non-subdivision set has size at least $6k^2$, there are at least $6k^2 - 5k^2 \ge k + 1 \ge p + 1$ free vertices in $S_p$. Therefore, we may set $\cQ^p = \cQ^{p-1}$ and set $v^p = v^{p-1}$, and note that $S_p$ is $(\cQ^p, p + 1)$-free.

Otherwise, there is
some path $P$ which intersects $S_p$ in at least $5k^2/(k+1) \ge 2k + 1$ vertices. Write these vertices in the intersection as $u_1, \ldots, u_{l}$ in the order they appear in $P$ with $l \ge 2k + 1$. If $P$ is the special path of $\cQ^{p-1}$, then $\cQ^p$ is simply formed by setting $v^p = u_1$ (the first intersection with $S_p$) and following $P$ to $v^p$. Thus we may assume that $P$ is not the special path of $\cQ^{p-1}$.
Now construct the following new paths. Follow the special path in $\cQ^{p-1}$ to $v^{p-1} \in S_z$. As each $S_q$ is a non-subdivision set, and these sets are part of a $(k,\ell)$-good family, by definition we have $S_z \rightarrow \ldots \rightarrow S_p$. Accordingly, we may go from $v^{p-1}$ to $u_{l}$ using free vertices, and then follow $P$ to its endpoint in $Y$. Call this path $P'$. Our new special path is formed by following $P$ to $u_1$ and setting $v^p = u_1$. Let $\cQ^p$ be the resulting family of paths. We have thus freed vertices $u_2, \ldots, u_{l-1}$ for a total of $l - 2 \ge 2k - 1 \ge k + 1 \ge p + 1$ vertices in $S_p$. Thus $\cQ^p$ satisfies the following properties:
\begin{itemize}
    \item $S_q$ is $(\cQ^p, q)$-free for all $1 \le q \le p-1$.
    \item $v^p \in S_p$ and $S_p$ is $(\cQ^p, p+1)$-free.
\end{itemize}
Indeed, the second item above is clear by construction, and the first item holds because $S_q$ is $(\cQ^{p-1}, q+1)$-free for $z \le q \le p - 1$, and the new path $P'$ uses precisely one free vertex from each of these sets. Thus $\cQ^p$ satisfies the desired properties. 

Finally, if the origin set $O$ was chosen as a subset of $S_t$, then note that $O$ remains invariant in this process and therefore is maintained as a free set of $k+1$ vertices in $S_t$. Therefore, we terminate this process with $v^* = v^{t-1}$ and $\cQ^* = \cQ^{t-1}$. On the other hand, if $O \subset \ter(P^s)$, then we repeat the above procedure to the set $S_t$ yielding $v^* = v^t$ and $\cQ^* = \cQ^{t}$. By the same argument we can guarantee at least $k + 1 \ge t + 1$ free vertices in $S_t$, completing the proof of the lemma.
\end{proof}

\subsection{Finishing the proof}\label{subsec:completing-2}
In the previous subsection we showed that there is a system $\cQ$ of pairwise vertex disjoint paths from $O$ to $Y \cup \{v\}$, for some $v \in \bigcup_{q=1}^t S_{j_q}$, that do not intersect any of the $S_i$'s in many vertices.
We shall use these \emph{free} vertices to extend the paths in $\cQ$ and obtain the desired pairwise vertex disjoint paths from $x_i$ to $y_i$, for $i \in [k]$, thus finishing the proof of \Cref{thm:main}.

Recall that we have assumed that $\ell \geq \Clbound{}$.
For each $i \in [k]$, let $z_i$ be the vertex in $O$ such that there is a path in $\cQ$ starting at $z_i$ and ending at $y_i$, and let $S'_{i} = S_{i} \setminus \bigcup \cQ$.
First, we need the following lemma, which says that the set $S'_{i_r}$ is, in a certain way, highly linked.
\begin{lemma}\label{lem:subdivisions_are_linked}
For any two disjoint sets of vertices $\{u_1, \ldots, u_k\}$ and $\{v_1, \ldots, v_k\}$ in $S'_{i_p}$, we can find pairwise vertex disjoint paths joining each $u_i$ to $v_i$ using only vertices of the subdivision sitting on $S'_{i_p}$ and avoiding the vertices of $\cQ$.
\end{lemma}
\begin{proof}
It follows from Lemmas~\ref{lem:subdivision-intersect}~and~\ref{lem: subd-paths-intr} that for any two vertices $u$ and $v$ in $S'_{i_p}$ there are at least $\ell - 2k - \Clemmabranches{} - 2\cdot \Clemmapaths{} \geq k$ vertices $w$ in $S'_{i_p} \setminus \{u_1, \ldots, u_k, v_1, \ldots, v_k \}$, such that we can go from $u$ to $w$ and from $w$ to $v$ via two paths of the subdivision and avoiding the vertices of $\cQ$.
Therefore, we can find a system of pairwise vertex disjoint paths by greedily choosing paths with the desired properties.
\end{proof}
Our goal now is to find a system of pairwise vertex disjoint paths $\cQ'$ joining $x_i$ to $z_i$, for each $i \in [k]$, that do not use any vertices of $\cQ$ or $Y$.
It is easy to see that if we find such a system then we are done: for each $i$ we can simply go from $x_i$ to $z_i$ via a path of $\cQ'$ and then from $z_i$ to $y_i$ using a path from $\cQ$, thus obtaining a path starting at $x_i$ and ending at $y_i$.
To achieve this goal, we have to consider two cases depending on whether $O$ is in $\ter(P^{ns})$ or in $\ter(P^{s})$.\\

\paragraph{\textbf{Case 1}: $O \subset \ter(P^{ns})$}

Using the fact that, for each $q \in [t]$, $|S'_{j_q}| \geq q$ (by Lemma~\ref{lem:non-sub-intersect}), and the property that $S'_{j_q} \rightarrow S'_{j_{q+1}}$ for each $q < t$, we can greedily find pairwise vertex disjoint paths from $x_{j_q}$ to $z_{j_q}$ using only vertices in $\bigcup_{q \in [t]} S'_{j_q}$.

On the other hand, to a find a path from $x_{i_p}$ to $z_{i_p}$ where $p \in [r]$, we shall use the property that $z_{i_p}$ has many in-neighbours in $S'_{i_r}$.
Indeed, as $O \subset \ter(P^{ns})$ we must have $i_rj_t \in E(H)$ in the auxiliary digraph $H$, by construction. \hyperlink{discarded}{Recall} that we have discarded those vertices in $S_{j_t}$ which have few in-neighbours in $S_{i_r}$, which means that
each of the vertices $z_{i_1}, \ldots, z_{i_r} \in O$ have at least
\[
    \ell/2 - \Clemmabranches{} \ge k
\]
in-neighbours in $S_{i_r}'$ (where we have applied \Cref{lem:subdivision-intersect}). Thus we may find $r$ distinct vertices $z_{i_1}', \ldots, z'_{i_r} \in S'_{i_r}$ such that $z_{i_1}' \rightarrow z_{i_1}, \ldots, z'_{i_r} \rightarrow z_{i_r}$.

Now, it follows from Lemma~\ref{lem:subdivision-intersect} and the definition of the auxiliary digraph $H$ that at least $\ell/2 - \Clemmabranches{} \geq 2k$ vertices in $S'_{i_p}$ have at least $\ell/2 - \Clemmabranches{} \geq 2k$ out-neighbours in $S'_{i_{p+1}}$, and hence, with a help of Lemma~\ref{lem:subdivisions_are_linked} which we use to arrive to vertices in $S'_{i_p}$ with high out-degrees in $S'_{i_{p+1}}$, we can greedily find pairwise vertex disjoint paths from $\{x_{i_1}, \ldots , x_{i_r}\}$ to $Z$, for some $Z \subseteq S'_{i_r} \setminus \{z'_1, ..., z'_r\}$, again using only vertices in $\bigcup_{p \in [r]} S'_{i_p}$.
Finally, using Lemma~\ref{lem:subdivisions_are_linked} we can appropriately link $Z$ to $z'_{i_1}, ..., z'_{i_r}$ to obtain a system of pairwise vertex disjoint paths from $x_{i_p}$ to $z'_{i_p}$ for every $p \in [r]$. Using the fact that $z'_{i_p} \rightarrow z_{i_p}$, we obtain the desired paths from $x_{i_p}$ to $z_{i_p}$ for each $p \in [r]$.

\paragraph{\textbf{Case 2}: $O \subset \ter(P^{s})$}
In this case, we must have $j_ti_r \in E(H)$, so each vertex in $S'_{j_{t}}$ has at least $\ell/2 - \Clemmabranches{} \geq 2k$ out-neighbours in $S'_{i_r}$. As before, we can greedily find pairwise vertex disjoint paths from $x_1, ..., x_k$ to $Z$, for some $Z \subseteq S'_{i_r} \setminus \{z_1,...,z_k\}$, using only vertices in $\bigcup_{i \in [k]} S'_{i}$.
Again, using Lemma~\ref{lem:subdivisions_are_linked} we can appropriately link $Z$ with $\{z_1, ..., z_{k}\}$ and obtain a system of pairwise vertex disjoint paths from $x_{i}$ to $z_{i}$ for every $i \in [k]$.

In each case, we have found the required collection of vertex disjoint directed paths linking $x_i$ to $y_i$ for each $i \in [k]$. This completes the proof of \Cref{thm:main}.\qed

\section{Constructions}\label{section:constructions}
\subsection{There exist $(2k-1)$-connected tournaments with large minimum out-degree which are not $k$-linked}\label{subsection:2k_necessary}

For all integers $k\geq 2$ and $m\geq 2k$, we construct a tournament $T$ on $n$ vertices which is $(2k-1)$-connected and whose minimum out-degree and in-degree is at least $m$, but which is not $k$-linked. 

Indeed, let $T$ be a tournament on vertex set $V = A \cup B \cup X \cup Y\cup C$, where $X = \{x_1,...,x_k\}$, $Y = \{y_1, ..., y_k\}$, $|C| = k-1$ and $|X| = |Y|= (n - 3k + 1)/2$ and whose edges are oriented in the following way.
\begin{enumerate}
    \item The edges within $A$, $B$, $C$ are oriented arbitrarily.
    \item The edges within $X$ and $Y$ are oriented so that $T[X]$ and $T[Y]$ form a $2m$-connected tournament.
    \item All edges are oriented from $Y$ to $X$, from $A$ to $C$, from $C$ to $B$, from $X$ to $C$, from $C$ to $X$, from $Y$ to $A$ and from $B$ to $X$.
    \item All edges are oriented from $A$ to $B$ except for edges between $x_i$ and $y_i$, for each $i\in [k]$.
    \item The edges between $A$ and $X$ are oriented in such a way that every vertex in $A$ sends at least $m$ out-edges to $X$ and $m$ in-edges to $X$. Similarly, the edges between $B$ and $Y$  are oriented in such a way that every vertex in $A$ sends at least $m$ out-edges to $Y$ and $m$ in-edges to $Y$. 
\end{enumerate}

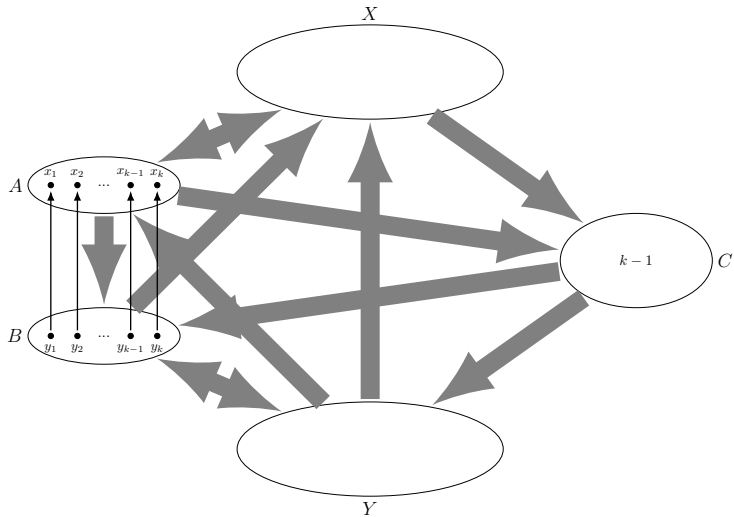
\begin{figure}[h]
    \centering
    \scalebox{0.5}{
        \input{tikz_notsobeautiful.tex}
    }
    \caption{A $(2k-1)$-connected tournament with large minimum in/out-degree that is not $k$-linked.}
    \label{fig:construction_2_k}
\end{figure}

We need to prove the following three properties of $T$. 
\begin{enumerate}
    \item $T$ is $(2k-1)$-connected.
    \item $T$ has minimum-out and in-degree at least $m/2$.
    \item There do not exist $k$ vertex-disjoint paths joining $x_i$ to $y_i$, for each $i\in [k]$. 
\end{enumerate}
\begin{proof}
Suppose $T$ is not $(2k-1)$-connected. Then there exists a subset $W\subset T$ of size at most $2k-2$ such that $T\setminus W$ is not connected. First, we show that $C\subseteq W$. If not, there must exist $z\in C\setminus W$ and it is not hard to see that every vertex can reach $z$ within $T\setminus W$ and every vertex can be reached from $z$, which is a contradiction. Hence, we may assume $C\subset W$. Note then that $|(A\cup B)\cap W|\leq k-1 $, which, in particular, implies that neither $A$ nor $B$ can be fully contained within $W$. Let $x_i\in A\setminus W$ and $y_{j_1},y_{j_2}\in B\setminus W$. It is easy to see that $x_1$ can reach every vertex in $T\setminus W$, since it can certainly reach $X\setminus W,C\setminus W$. Moreover, it can reach either $y_{j_1}$ or $y_{j_2}$, and then via one of these vertices, it can reach $Y\setminus W$. A similar argument shows that any vertex can reach $x_i$, which is a contradiction.

It is easy to see that every vertex $x$ has $d^{+}(x)\geq m/2$. 
Finally, we need to show there do not exist $k$ vertex disjoint paths joining $x_i$ to $y_i$. Observe that any path from $x_i$ to $y_i$ can not use any vertex of $(B\setminus{y_i})\cup (A\setminus{x_i})$ and therefore it must use a vertex of $C$. But since $|C|<k$ this is not possible. 

\end{proof}

\subsection{There exist $(2.5k-1)$-connected tournaments which fail to be $k$-linked}\label{subsection:construction}

We shall now show that for each $k \geq3$ and any sufficiently large $n$ there exist $(5k-1)$-connected tournaments on $n$ vertices that are not $2k$-linked, which shows that the minimum out-degree condition in our theorem is necessary.

Let $T$ be a tournament on vertex set $V = X \cup Y \cup S \cup W$, where $X = \{x_1,...,x_k\}$, $Y = \{y_1, ..., y_k\}$, $|S| = 4k-1$ and $|W| = n - 6k + 1$, and whose edges are oriented in the following way.
\begin{enumerate}
    \item The edges inside each of $X$, $Y$, $S$, $W$, and between $S$ and $W$ are oriented in such a way that $T[S \cup W]$ is $(5k-1)$-connected (for large enough $n$, a random configuration of edges in $S \cup W$ will have this property), and both $T[X]$ and $T[Y]$ are strongly connected.
    \item All edges are oriented from $X$ to $Y$, from $Y$ to $W$, and from $W$ to $Y$.
    \item For every $i$, all the edges are oriented from $x_i$ to $S$ except for the edge between $x_i$ and $y_i'$ for some unique vertex $y'_i$.
    \item For every $i$, all the edges are oriented from $S$ to $y_i$ except for the edge between $x'_i$ and $y_i$ for some unique vertex $x'_i \not\in \{y'_1, ..., y'_k\}$.
\end{enumerate}

\begin{figure}[h]
    \centering
    \scalebox{0.5}{
        \input{tikz_beautiful_construction.tex}
    }
    \caption{Example of a $(5k-1)$-connected tournament that is not $2k$-linked.}
    \label{fig:construction_2_5_k}
\end{figure}
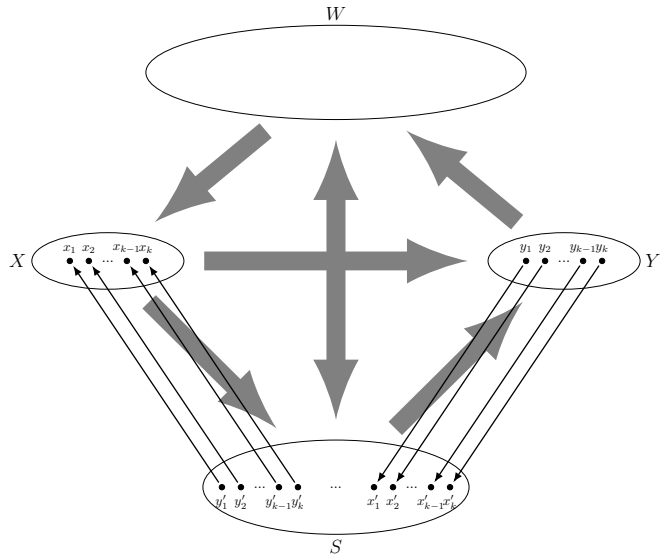

\begin{claim}\label{claim_T_is_not_linked}
    $T$ is not $2k$-linked.
\end{claim}
\begin{proof}
    Observe that for each $i$ any path joining $x_i$ to $y_i'$ must use an extra vertex from $X \cup S \cup Y$.
    The same holds from paths going from $x_i'$ to $y_i$.
    Hence, any system of disjoint paths joining $x_i$ to $y_i'$ and $x_i'$ to $y_i$, for every $i$, uses at least $6k$ vertices in $X \cup Y \cup S$.
    But this cannot happen, as by construction $|X \cup Y \cup S| = 6k-1$.
\end{proof}

\begin{claim}\label{claim_T_highly_connected}
    $T$ is $(5k-1)$-connected.
\end{claim}
\begin{proof}
    Let $T'$ be an tournament obtained by removing any $5k-2$ vertices from $T$.
    We shall show that $T'$ is still connected.
    Let us write $X'$, $Y'$, $W'$, $S'$ for $X \cap T'$, $Y \cap T'$, $W \cap T'$, $S \cap T'$, respectively.
    By construction $W \cup S$ is $(5k-1)$-connected, therefore $W' \cup S'$ is still connected and hence every vertex in $T'$ can be reach via a directed path from $W' \cup S'$. Therefore, it remains to show that (1) there is a path between any vertex in $X$ and $W' \cup S'$, and (2) a path between $W' \cup S'$ and any vertex in $Y$. We will only prove (1) as the proof of (2) is symmetrical.
    
    Take any $x_i \in X'$.
    Observe that if $Y' \neq \emptyset$ or $S' \neq \{y_i'\}$ then we can easily find a path from $x_i$ to $W' \cup S'$.
    We can therefore assume that $Y' = \emptyset$ and $S' = \{y_i\}$, and therefore $X' = X$.
    By construction $x_i$ is the only vertex in $X'=X$ which does not send an out-edge to $y_i'$, hence $x_i$ can reach $y_i'$ using any other vertex in $X'=X$.
    
\end{proof}
\section{Final remarks}\label{sec:final_remarks}
An analysis of our methods shows that there is an absolute constant $C > 0$ such that any $(2k+1)$-connected tournament with minimum out-degree at least $Ck^{31}$ is $k$-linked. 
We remark that we did not make a strong effort to optimize the power of $k$ in the minimum out-degree condition. 
While we believe that we could bring its value down, we were unable to obtain a linear bound, which we conjecture is the truth. 
\begin{conjecture}\label{conj:linear_degree}
    There exists a constant $C>0$ such that every $(2k+1)$-connected tournament with minimum out-degree at least $Ck$ is $k$-linked.
\end{conjecture}
In Subsection~\ref{subsection:2k_necessary}, we showed that one cannot replace $2k+1$ by $2k-1$ in Theorem~\ref{thm:main}, as there exist arbitrarily large tournaments which are $(2k-1)$-connected with large minimum out and in-degree, but fail to be $2k$-linked.
We have not ruled out the possibility that the connectivity condition can be relaxed to $2k$, however.
It is therefore natural to ask the following.
\begin{question}\label{question:2k_or_2kplus1}
    Does Theorem~\ref{thm:main} still hold if we replace $2k+1$ by $2k$?
\end{question}

Note that an affirmative answer to this question would completely resolve Conjecture~\ref{conj:pokrovskiy} in a stronger form, in the sense of not additionally requiring large minimum in-degree.




\bibliographystyle{amsplain}
\bibliography{main.bib}
\end{document}

%% file: tikz_notsobeautiful.tex
\begin{tikzpicture}
    \newcommand{\cVertexSize}{0.5mm}
    \newcommand{\cFatArrowColor}{gray}              
    \newcommand{\cThinArrowColor}{black}              
    \newcommand{\cMiddleBlobsVerticalOffset}{6}     
    \node[ellipse, minimum height=1.5cm,minimum width=4cm, draw=black, label=left:\Large{$A$}](blobA) at (-7, 2) {...};
    
    \node[ellipse, minimum height=1.5cm,minimum width=4cm, draw=black, label=left:\Large{$B$}](blobB) at (-7,-2) {...};
    
    \node[ellipse, draw=black, minimum height=2.5cm,minimum width=7cm, label=\Large{$X$}](blobX) at (0,5) {};
    
    \node[ellipse, minimum height=2.5cm,minimum width=7cm, draw=black, label=below:\Large{$Y$}](blobY) at (0,-5) {};
    
    \node[ellipse, minimum height=2.5cm,minimum width=4cm, draw=black, label=right:\Large{$C$}](blobC) at (7,0) {$k-1$};
    
    \begin{scope}[->, >=latex, color=\cFatArrowColor, line width=5mm, shorten >=2pt,shorten <=2pt]
        \draw (blobX) edge (blobC);
        \draw (blobC) edge (blobY);
        \draw (blobY) edge (blobX);
        \draw (blobA) edge (blobC);
        \draw (blobC) edge (blobB);
        \draw (blobY) edge (blobA);
        \draw (blobB) edge (blobX);
        \draw (blobA) edge (blobB);
        \draw [<->] (blobA) edge (blobX);
        \draw [<->] (blobB) edge (blobY);
    \end{scope}
    
    \node[circle, draw=black, fill=black, inner sep=\cVertexSize, label=\small{$x_1$}](x1) at (-7-1.4, 2) {};
    \node[circle, draw=black, fill=black, inner sep=\cVertexSize, label=\small{$x_2$}](x2) at (-7 - 0.7, 2) {};
    \node[circle, draw=black, fill=black, inner sep=\cVertexSize, label=\small{$x_{k-1}$}](xk1) at (- 7 + 0.7, 2) {};
    \node[circle, draw=black, fill=black, inner sep=\cVertexSize, label=\small{$x_{k}$}](xk) at (- 7 + 1.4, 2) {};
    
    \node[circle, draw=black, fill=black, inner sep=\cVertexSize, label=below:\small{$y_1$}](y1) at (-7-1.4, -2) {};
    \node[circle, draw=black, fill=black, inner sep=\cVertexSize, label=below:\small{$y_2$}](y2) at (-7 - 0.7, -2) {};
    \node[circle, draw=black, fill=black, inner sep=\cVertexSize, label=below:\small{$y_{k-1}$}](yk1) at (- 7 + 0.7, -2) {};
    \node[circle, draw=black, fill=black, inner sep=\cVertexSize, label=below:\small{$y_{k}$}](yk) at (- 7 + 1.4, -2) {};

    \begin{scope}[->, >=Latex, color=\cThinArrowColor, line width=1pt, shorten >=2pt,shorten <=2pt]
        \draw (y1) edge (x1);
        \draw (y2) edge (x2);
        \draw (yk1) edge (xk1);
        \draw (yk) edge (xk);
    
    \end{scope}
\end{tikzpicture}

%% file: tikz_beautiful_construction.tex
\begin{tikzpicture}
    \newcommand{\cVertexSize}{0.5mm}
    \newcommand{\cFatArrowColor}{gray}              
    \newcommand{\cThinArrowColor}{black}              
    \newcommand{\cMiddleBlobsVerticalOffset}{6}     
    \newcommand{\drawVertex}[4]{
        \node[circle, draw=black, fill=black, inner sep=\cVertexSize, label=\small{#1}](#2) at (#3, #4) {};
    }
    \newcommand{\drawVertexBelow}[4]{
        \node[circle, draw=black, fill=black, inner sep=\cVertexSize, label=below:\small{#1}](#2) at (#3, #4) {};
    }
    \node[ellipse, minimum height=1.5cm,minimum width=4cm, draw=black, label=left:\Large{$X$}](blobX) at (-\cMiddleBlobsVerticalOffset, 0) {...};
    
    \node[ellipse, minimum height=1.5cm,minimum width=4cm, draw=black, label=right:\Large{$Y$}](blobY) at (\cMiddleBlobsVerticalOffset,0) {...};
    
    \node[ellipse, draw=black, minimum height=2.5cm,minimum width=10cm, label=\Large{$W$}](blobW) at (0,5) {};
    
    \node[ellipse, minimum height=2.5cm,minimum width=7cm, draw=black, label=below:\Large{$S$}](blobS) at (0,-6) {...};
    
    \begin{scope}[->, >=latex, color=\cFatArrowColor, line width=5mm, shorten >=15pt,shorten <=15pt]
        \draw (blobX) edge (blobY);
        \draw (blobY) edge (blobW);
        \draw (blobW) edge (blobX);
        \draw [<->] (blobW) edge (blobS);
        \draw (blobX) edge (blobS);
        \draw (blobS) edge (blobY);
    \end{scope}
    
    \drawVertex{$x_1$}{x1}{-\cMiddleBlobsVerticalOffset - 1}{0}
    \drawVertex{$x_2$}{x2}{-\cMiddleBlobsVerticalOffset - 0.5}{0}
    \drawVertex{$x_{k-1}$}{xk1}{-\cMiddleBlobsVerticalOffset + 0.5}{0}
    \drawVertex{$x_{k}$}{xk}{-\cMiddleBlobsVerticalOffset + 1}{0}
    
    \drawVertex{$y_1$}{y1}{\cMiddleBlobsVerticalOffset - 1}{0}
    \drawVertex{$y_2$}{y2}{\cMiddleBlobsVerticalOffset - 0.5}{0}
    \drawVertex{$y_{k-1}$}{yk1}{\cMiddleBlobsVerticalOffset + 0.5}{0}
    \drawVertex{$y_{k}$}{yk}{\cMiddleBlobsVerticalOffset + 1}{0}
    
    \drawVertexBelow{$y_1'$}{y1p}{-3}{-6}
    \drawVertexBelow{$y_2'$}{y2p}{-2.5}{-6}
    \node at (-2.0,-6) {...};
    \drawVertexBelow{$y_{k-1}'$}{yk1p}{-1.5}{-6}
    \drawVertexBelow{$y_{k}'$}{ykp}{-1}{-6}
    \drawVertexBelow{$x_1'$}{x1p}{1}{-6}
    \drawVertexBelow{$x_2'$}{x2p}{1.5}{-6}
    \node at (2.0,-6) {...};
    \drawVertexBelow{$x_{k-1}'$}{xk1p}{2.5}{-6}
    \drawVertexBelow{$x_{k}'$}{xkp}{3}{-6}
    
    \begin{scope}[->, >=Latex, color=\cThinArrowColor, line width=1pt, shorten >=2pt,shorten <=2pt]
        \draw (y1p) edge (x1);
        \draw (y2p) edge (x2);
        \draw (yk1p) edge (xk1);
        \draw (ykp) edge (xk);
        
        \draw (y1) edge (x1p);
        \draw (y2) edge (x2p);
        \draw (yk1) edge (xk1p);
        \draw (yk) edge (xkp);
    \end{scope}
\end{tikzpicture}